\documentclass[12pt]{amsart}
\usepackage{fullpage}
\usepackage{amscd, bbm, amsaddr}
\usepackage{amssymb,mathabx}
\usepackage{mathrsfs,wasysym, tikz, tikz-cd, extpfeil}
\usetikzlibrary{matrix, arrows}
\usepackage{amsmath, enumitem}
\usepackage{geometry}

\usepackage{color}
\usepackage{hyperref}

\newtheorem{Theorem}{Theorem}[section]
\newtheorem{Lemma}[Theorem]{Lemma}
\newtheorem{Proposition}[Theorem]{Proposition}
\newtheorem{Corollary}[Theorem]{Corollary}

\theoremstyle{definition}
\newtheorem{Definition}[Theorem]{Definition}

\theoremstyle{remark}
\newtheorem{Remark}[Theorem]{Remark}

\newtheorem{Example}[Theorem]{Example}

\usepackage{enumitem}

\newcommand{\C}{\mathbb{C}}
\newcommand{\Nbb}{\mathbb{N}}
\newcommand{\Q}{\mathbb{Q}}

\newcommand{\g}{\mathfrak{g}}

\newcommand{\gl}{\mathfrak{gl}}
\newcommand{\h}{\mathfrak{h}}
\renewcommand{\t}{\mathfrak{t}}
\renewcommand{\l}{\mathfrak{l}}
\renewcommand{\sl}{\mathfrak{sl}}

\newcommand{\z}{\mathfrak{z}}
\newcommand{\m}{\mathfrak{m}}

\newcommand{\J}{\mathfrak{J}}

\renewcommand{\a}{\mathfrak{a}}

\newcommand{\orb}{\mathfrak{O}}

\newcommand{\N}{\mathcal{N}}
\renewcommand{\O}{\mathcal{O}}
\newcommand{\U}{\mathcal{U}}

\newcommand{\pr}{\operatorname{pr}}
\newcommand{\diff}{\operatorname{d}}

\newcommand{\Ad}{\operatorname{Ad}}
\newcommand{\tr}{\operatorname{tr}}

\newcommand{\Lie}{\operatorname{Lie}}
\newcommand{\spn}{\operatorname{span}}
\newcommand{\image}{\operatorname{Im}}

\newcommand{\reg}{{\operatorname{reg}}}

\newcommand{\Spec}{\operatorname{Spec}}

\newcommand{\Ind}{\operatorname{Ind}}

\newcommand{\GL}{\operatorname{GL}}
\newcommand{\SL}{\operatorname{SL}}

\newcommand{\Sp}{\operatorname{Sp}}

\newcommand{\PSL}{\operatorname{PSL}}

\newcommand{\id}{\operatorname{id}}

\newcommand{\Res}{\operatorname{Res}}
\newcommand{\Inn}{\operatorname{Inn}}
\newcommand{\se}{\operatorname{se}}

\title[\'Etale geometry of closures of Jordan classes]{\'Etale geometry of closures of Jordan classes}
\author[F. Ambrosio]{Filippo Ambrosio}
\address{FSU Jena, Fakult\"at f\"ur Mathematik und Informatik,\\ Ernst-Abbe-Platz 2, 07743 Jena (Germany)}
\email{filippo.ambrosio@uni-jena.de}

\begin{document}

\begin{abstract}
Let $G$ be a connected reductive algebraic group with  simply connected derived subgroup.
Over the complex numbers there exists a local method to study the geometric properties of a point $g$ in the closure of a Jordan class of $G$ in terms of Jordan classes of a maximal rank reductive subgroup $M \leq G$ depending on the point $g$, and further to the closures of certain decomposition classes in $\Lie M$.
We adapt this method to the case of an algebraically closed field of characteristic $p$, and we give sufficient restrictions on $p$ for it to hold.
\end{abstract}

\maketitle


\subsection*{Acknowledgments}
The author would like to thank Giovanna Carnovale for suggesting the problem and Mauro Costantini, Francesco Esposito, George McNinch and Oksana Yakimova for helpful and stimulating discussions.
The author would also like to thank the anonymous referees for carefully reading the manuscript and for their suggestions.
This research work was mainly conducted at FSU Jena. 



\section{Introduction} \label{s_intro}
The conjugacy action of a connected reductive group $G$ on itself is a central problem in the study of algebraic Lie theory.
Suppose that the base field $k$ is algebraically closed of arbitrary characteristic $p$.
Then the multiplicative Jordan-Chevalley decomposition suggests a natural way to gather conjugacy classes into irreducible locally closed subsets of $G$, called  Jordan classes.
Elements in the same Jordan class have semisimple part  allowed to range over a subset of points with $G$-conjugate connected centralizer, while their unipotent parts must be conjugate in $G$.
This ensures that Jordan classes are finitely many smooth varieties consisting of conjugacy classes of the same dimension.
These varieties appear in many topics in representation theory of algebraic groups: they were introduced by Lusztig in \cite{LusztigICC} as the loci where character sheaves are locally constant.
Jordan classes also play a crucial role in the description of sheets of $G$, i.e., the irreducible components of subsets of $G$ consisting of elements with equidimensional classes \cite{ACE, ACEbad}.
We shortly illustrate the connection: the set of Jordan classes of $G$ is partially ordered by the relation $J_1 \preceq J_2$ if and only if $J_1 \subset \overline{J_2}^\reg$, where the regular closure $\overline{J_2}^\reg$ is the subset of all elements of $\overline{J_2}$ with centralizer of minimal dimension.
Sheets of $G$ are parametrized as $\overline{J}^\reg$ where $J$ runs over all Jordan classes $J \subset G$ which are maximal with respect to $\preceq$.
Moreover, Lusztig defined in \cite{LusztigCCRG} a map from $G$ to the set of irreducible representations of its Weyl group by means of truncated induction of Springer's representations for trivial local systems.
Each fibre of this map is called a stratum and the irreducible components  of a stratum are the sheets contained therein, see \cite{CarnoBul}.
All these instances give evidence to the importance of producing methods to study properties of Jordan classes, as well as of their regular closures.

The Lie algebra $\g$ of $G$ admits an additive Jordan-Chevalley decomposition, by means of which one can define objects analogous to Jordan classes, called decomposition classes or packets, see \cite{BK79, BroerDecoVar, BroerLectures, PremSt, Sp_bad}. 
Under the assumption that $p$ is very good for $G$ (see section \ref{s_intro}), decomposition classes are characterized as subsets of elements whose centralizer in $\g$ is $G$-conjugate (\cite[Theorem 3.7.1]{BroerLectures}); moreover, they are smooth (\cite[Corollary 3.8.1]{BroerLectures}), just like Jordan classes.

Therefore, possibly with some restrictions on $p$, it is expected that the behaviour of decomposition classes mirrors somehow the situation in $G$.
In particular, in \cite{ACE} an approach to study the geometric properties of the closure of a Jordan class was elaborated in the case in which $G$ is complex semisimple simply connected.

The principal aim of this paper is to adapt the method in \cite{ACE} to the wider generality of the case $p \neq 0$.
We briefly describe the procedure.
 For a Jordan class $J \subset G$ and an element with Jordan decomposition $su \in \overline{J}$, there is a smooth equivalence between the pointed varieties $(\overline{J}, su)$ and $(X, su)$, where $X$ is given by a  union of closures of Jordan classes in the centralizer $C_G(s) \leq G$ that are determined by the data of $J$ and $su$. 
In \cite{ACE} this was possible thanks to the construction of an \'etale neighbourhood of $s$ in the style of Luna \cite{Luna}; in the current manuscript we use a characteristic-free version of Luna's result by Bardsley--Richardson \cite{BR}.
We thus obtain a local description of $\overline{J}$ and $\overline{J}^\reg$ (see Theorem \ref{th_ACE}) as a natural extension of the first main result of \cite{ACE} so that the attention is drawn particularly on closures of Jordan classes around unipotent elements.

As explained always in \cite[\S 5]{ACE}, if $k = \C$, the exponential map induces an analytic isomorphism between neighbourhoods of the unipotent variety of a reductive group $G$ and of the nilpotent cone of $\g$.
Generalizations of the exponential map to the case $p > 0$ can be found in \cite{Sobaje, SobTay}, but in our scope the best substitute seems to be a log-like map \`a la Bardsley--Richardson \cite{BR, McN_NOgood}.
The restriction of a log-like map to its \'etale locus allows to reduce the study of the geometry of a unipotent element in the closure of a Jordan class in $G$ to the study of a nilpotent element in the closure of a decomposition class in $\g$, see Corollary \ref{cor_se}.
This technique is the main novelty of the paper and is affected by typical phenomena of positive characteristic, some of which are illustrated in the final part of the manuscript. 

\subsection{Structure of the paper}
In section \ref{s_notation} we fix notation and recall the concepts from the theory of algebraic varieties and algebraic groups needed in the rest of the exposition.
Section \ref{s_jordan} introduces decomposition and Jordan classes together with their main properties; in particular, we recall their role in the description of sheets of adjoint orbits, resp. of conjugacy classes.
The first results of the manuscript are Theorem \ref{th_ACE} and its corollaries, contained in section \ref{s_redunip}.
Since the main ideas therein come from \cite{ACE}, we have tried to avoid repetitions and to highlight only the arguments of the proofs which diverge from the case $k = \C$.
Finally, section \ref{s_br} introduces the notion of a log-like map and recalls the main examples from \cite{BR, McN_NOgood}.
The new aspect in the approach of this paper is illustrated  in part \ref{ss_compatible} where compatibility of a log-like map with the Jordan stratification on an \'etale neighbourhood of the unipotent variety is explained, leading to the second main result of the paper, Corollary \ref{cor_se}.
The precise extents to which the methods of \cite{ACE} can be extended in the case $p > 0$ are explained in \ref{sss_summary}.
The manuscript closes with some examples and applications: in particular, we give a new proof of the fact that sheets of $\GL_n(k)$ are smooth, already mentioned in \cite{LusztigICC}. 

\section{Notation and recollections} \label{s_notation}
Throughout the paper we denote by $k$ an algebraically closed field and we set $p = 0$ or a prime to be the characteristic of $k$.
We will work with algebraic groups and varieties defined over $k$, omitting the ground field when clear from the context.
For an algebraic group $H$, we denote by $Z(H)$ its centre and by $H^\circ$ the connected component of $H$ containing the identity $e$ of $H$.

\subsection{Notions from algebraic geometry}
For a variety $X$, its algebra of regular functions will be denoted $k[X]$ and for $f \in k[X]$ we set $\mathcal{Z}(f) = \{ x \in X \mid f(x) = 0 \}$.
Let $X,Y$ be varieties and let $\varphi \colon X \to Y$ be a morphism of varieties.
We denote with $\varphi^* \colon k[Y] \to k[X]$ its pull-back.
For $x \in X$, we write $\diff_x \varphi \colon T_x X \to T_{\varphi(x)} Y$ for the differential of $\varphi$ at $x$.
For $x \in X$, the morphism $\varphi$ is said to be \emph{\'etale at $x$}  if $\varphi$ is smooth and unramified at $x$.
When $x \in X$ and $\varphi(x) \in Y$ are smooth points, then the map $\varphi$ is \'etale at $x$ if and only if $\diff_x \varphi \colon T_x X \to T_{\varphi(x)} Y$ is an isomorphism.
The \'etale locus of $\varphi$ is the subset of points of $X$ where $\varphi$ is \'etale and it is an open subset of $X$.

Fix two points $x \in X, y \in Y$. Following \cite[\S 1.7]{Hess}, we say that the pointed variety $(X,x)$ is \emph{smoothly equivalent} to the pointed variety $(Y,y)$ if there exist a variety $Z$, a point $z \in Z$ and two morphisms
$\varphi_X \colon  Z \to X$ and $\varphi_Y \colon  Z \to Y$ s.t. $\varphi_X(z) = x, \varphi_Y(z) = y$ and $\varphi_X, \varphi_Y$ are  smooth at $z$.
We write briefly $(X,x) \sim_{\se} (Y,y)$.
This defines an equivalence relation on the set of pointed varieties which preserves the property of the variety being unibranch, normal, Cohen-Macaulay or smooth at the fixed point.
This fact follows from the local structure of a smooth morphism (see \cite[Proposition 3.24 (b)]{Milne}): the proof for normality can be found for example in \cite[Lemma 3.2]{XS}, but the argument can be adapted to all of the above properties.

\subsection{Group actions on varieties}
Let $X$ be a variety with a rational $H$-action:
$(h,x) \mapsto h \cdot x$ for $h \in H, x \in X$; we write $H \cdot x$ for the $H$-orbit of $x$ in $X$.
If $Y \subset X$, we set $H \cdot Y = \bigcup_{y \in Y} H \cdot y$.
For $n \in \Nbb$, we define the level set $X_{(n)} := \{ x \in X \mid \dim H \cdot x = n \}$;
the variety $X$ is partitioned into its level sets, which are locally closed. 
For $Y \subset X$, we write $Y^\reg := Y \cap X_{(n)}$ where $n$ is maximal with the property that $Y \cap X_{( n)} \neq \varnothing$.
Irreducible components of nonempty level sets of $X$ are called the \emph{sheets} of $X$ (for the action of $H$).

Suppose now that $H^\circ$ is reductive and that $X$ is affine.
The subalgebra of invariants $k[X]^H$ is finitely generated, and we denote by $X/\!/H := \Spec (k[X]^H)$ the categorical quotient of $X$ by $H$; this affine variety comes with a  surjective morphism of varieties $\pi_{X, H} \colon X \to X/\!/H$ (or just $\pi_H$ when there is no confusion) induced by the inclusion of algebras $k[X]^H \to k[X]$. 
A short summary on properties in a characteristic-free setting of $X/ \! /H$ and $\pi_{X,H}$ can be found in \cite[\S 2]{BR}, for a complete and classical reference see \cite{MumGit}.
We recall that in this generality, if $Y \subset X$ is closed and $H$-stable, then the closed subset $\pi_{X, H}(Y) \subset X/\!/ G$ is not necessarily isomorphic to $Y/\!/G$, and we only have a closed embedding $Y/\!/G \to \pi_{X, H}(Y)$.

Agreeing with terminology of \cite{Luna},  a subset $S \subset X$ is said to be $\pi_{X}$-saturated (or just saturated when the map is clear from the context) if $S = \pi_{X}^{-1}(\pi_{X}(S))$; it can be shown that $S$ is saturated if and only if, for all $x \in S$ also $y \in S$ for all $y \in X$ s.t. $\pi_X(x) = \pi_X(y)$, equivalently s.t. $\overline{H \cdot x} \cap \overline{H \cdot y} \neq \varnothing$.

We recall the construction of the associated fibre bundle.
Let $K \leq H$ be an algebraic subgroup with $K^\circ$ reductive and let $X$ be an affine $K$-variety.
Consider the product $H \times X$, we let $K$ act freely on $(h,x) \in H \times X$ via $(hg^{-1}, g\cdot x), g \in K$.
All $K$-orbits in $H \times X$ being closed, the quotient space agrees with the categorical quotient
 $(H \times X) /\!/ K$, which in this case is  denoted $H \times^K X$ and the class of the element $(h,x) 
 \in H \times X$ is denoted $h * x$.
The group $H$ defines a left action on $H \times^K X$ via left multiplication on the first component and there is a natural  $H$-equivariant projection $H \times^K X \to H/K$ with fibre isomorphic to $X$.
The isomorphism of $k$-algebras $k[X]^K \cong (k[H] \otimes k[X])^H$ (see \cite[\S 5.3.3]{BR}) implies in particular  that closed $H$-stable subsets of $H \times^K X$ correspond bijectively to closed $K$-stable subsets of $X$.
For further information on homogeneous fibre bundles, the reader is referred to \cite{Se}, a summary of the needed results can also be found in \cite[\S 5]{BR}.

\subsubsection{Conjugacy and adjoint action}
We will be concerned with two actions in particular.
The algebraic subgroup $K \leq H$ acts on $H$ via conjugation; the $K$-conjugacy class of the element $h \in H$ will be denoted 
$\O^K_h = \{ ghg^{-1} \mid g \in K\} \subset H$.
For a subset $S \subset H$ and $g \in K$ we will write briefly $g S g^{-1} = \{gsg^{-1} \mid s \in S \} \subset H$ and $K \cdot S = \bigcup_{s \in S} \O^K_s$. 
We also write $C_K(S) = \{g \in K \mid g \cdot s = s \mbox{ for all } s \in S \}$; if $S = \{h\}$ is a singleton, we write briefly $C_K(\{h\}) = C_K(h) = K_h$.

For $\h := \Lie H$, we will consider the adjoint action of $H$ on $\h$ defined by $\Ad \colon H \to \GL(\h), h \mapsto \Ad(h) = \diff_e \Inn(h)$, where $\Inn(h) \colon H \to H, g \mapsto hgh^{-1}$.
For $x \in \h, h \in H$, we use notations $h \cdot x := \Ad(h)(x)$ and $\orb^H_x := \{h \cdot x \mid h \in H \}$.
Moreover, we write $H_x := C_H(x) = \{h \in H \mid h \cdot x = x\}$.

\subsection{Reductive groups and their Lie algebras}
We denote with $G$ a connected reductive algebraic group over $k$.
We fix a Borel subgroup $B \leq G$ and a maximal torus $T \leq B$.
We denote with $W$ the Weyl group of $G$ with respect to $T$, with $\Phi$  the root system of $G$ with respect to $T$ and with $\Delta = \{ \alpha_1, \dots, \alpha_\ell\} \subset \Phi$ the system of simple roots  corresponding to $B$. 
For $\alpha \in \Phi$, we fix a root subgroup $U_\alpha$ whose elements will be sometimes described as the image of a suitable group morphism $u_\alpha \colon k \to G$.
With an abuse of terminology, a Levi subgroup of $G$ is understood to be the Levi factor of a parabolic subgroup of $G$.

The Lie algebra of $G$ will be denoted $\g$ and for a closed subgroup $M \leq G$, we will use the corresponding fraktur lower case letter $\m$ to denote its Lie algebra. 
It is known that $\g$ is a restricted Lie algebra (or Lie $p$-algebra, as defined by Jacobson \cite{Jacobson}).
An element of $\g$ is said to be semisimple (resp. nilpotent) if it belongs to the Lie algebra of a  torus (resp. unipotent subgroup) of $G$.
Any element $x \in \g$ admits a unique Jordan decomposition as a sum $x = x_s + x_n$ with $x_s \in \g$ semisimple, $x_n \in \g$ nilpotent and $[x_s , x_n] = 0$, see \cite{Bo}.
Moreover if $G \leq \GL_n(k)$ for some $n$, then the Jordan decomposition of an element in $\g$ agrees with its Jordan decomposition  as a linear endomorphism of $k^n$.

When we write $g = su \in G$ or $x = x_s + x_n \in \g$ we implicitly assume that
$su$ ($x_s +x_n$, respectively) is the Jordan decomposition of $g$ ($x$, respectively), with
$s$ semisimple and $u$ unipotent ($x_s$ semisimple and $x_n$ nilpotent, respectively).

Suppose $\Phi$ is irreducible, we say that $p$ is bad for $\Phi$ if $p =2$ and $\Phi$ is not of type $\sf A$, $p=3$ and $\Phi$ is not classical, or $p=5$ and $\Phi$ is of type $\sf E_8$. Otherwise, we say that $p$ is good for $\Phi$. If  $p$ is good for $\Phi$ and $p \not\mid (d +1)$ for $\Phi$ of type ${\sf A}_d$, we say that $p$ is very good for $\Phi$.
Analogously, we say that $p$ is torsion for $\Phi$ if $p = 2$ and $\Phi$ is not of type $\sf A$ or $\sf C$, $p = 3$ and $\Phi$ is exceptional not of type $\sf G_2$, or $p = 5$ and $\Phi$ is of type $\sf E_8$.
The following terminology is nowadays standard, see for example \cite[\S 14]{MalleTesterman}.
We say that $p$ is \emph{good} (resp. \emph{very good}) for $G$ if it is good (resp. very good) for all the irreducible components of the root system of $G$. 
We say that $p$ is \emph{separably good} for $G$ if $p$ is good for $G$ and $[G,G]$ is separably isogenous to its simply connected cover.
Finally 
$p$ is said to be \emph{torsion} for $G$ if $p$ is torsion for some 
irreducible component of $[G,G]$ or $p$ divides the order of the fundamental group of $G$, see \cite{St} for a detailed analysis of torsion primes.
The condition $p$ very good is stronger than $p$ separably good, which in turn is stronger than $p$ good.

We denote by $\U(G)$ the unipotent variety of $G$ and by $\N(\g)$ the nilpotent cone of $\g$.
A Springer isomorphism for $G$ is a $G$-equivariant isomorphism $\U(G) \cong \N(\g)$; for an overview on this topic see \cite[\S 6.20]{HuCCSAG} and \cite[\S 1.15]{Ca2}.
The existence of a Springer isomoprhism is guaranteed, for example, under the assumption that $p$ is separably good for $G$, see \cite[\S 4]{McN_optimal} and \cite[\S 2.4]{McNTes}, building on results from \cite{Springer_unipotent} and \cite[\S 9.3]{BR}.

\subsubsection{Centralizers of semisimple elements}
Let $s \in T$, then the structures of $G_s$ and $G_s^\circ$ are known, see \cite{St}; in particular, 
$G_s^\circ  = \langle T , U_{\pm \alpha} \mid \alpha \in \Phi_s \rangle$, where $\Phi_s := \{\alpha \in \Phi \mid \alpha(s) = 1\}$.
Such a group is reductive and sometimes called a standard pseudo-Levi subgroup, see \cite{SommersBC}.
For $x_s \in \t$ the paper \cite{St} contains analogous results about $G_{x_s}$ and $G_{x_s}^\circ$, in particular  $G_{x_s}^\circ  = \langle T , U_{\pm \alpha} \mid \alpha \in \Phi_{x_s} \rangle$, where 
$\Phi_{x_s} = \{\alpha \in \Phi \mid \diff_e \alpha(x_s) = 0\}$.
Moreover, we will use the following results, both due to Steinberg:
\begin{enumerate}
\item If $[G,G]$ is simply connected, then $G_s^\circ = G_s$ for all $s \in T$ (\cite{St}, see also \cite[Theorem 3.5.6]{Ca2}).
\item  $p$ is not a torsion prime for $G$ if and only if $G_{x_s}^\circ = G_{x_s}$ for all $x_s \in \t$ (\cite[Theorem 3.14]{St}).
\end{enumerate}

We recall Slodowy's criterion \cite[\S 3.5, Proposition]{Sl}, see also \cite[VI, 1, 7, Propostion 24]{Bour4}:
a root subsystem $\Psi \subset \Phi$ admits a base  contained in $\Delta$ if and only if it is rationally closed, i.e. $\left( \spn_\Q \Psi  \right) \cap \Phi = \Psi$.
Applying this result to the cases $\Psi = \Phi_s$ and $\Psi = \Phi_{x_s}$ we obtain that
$G_{s}^\circ$ (resp. $G_{x_s}^\circ$) is a Levi subgroup if and only if $\Phi_s$ (resp. $\Phi_{x_s}$) is rationally closed.
In particular, if $p$ is good for $G$, then $G_{x_s}^\circ$ is a Levi subgroup for all $x_s \in \t$, see \cite[Proposition 2.7.1]{Letellier}.

Since semisimple orbits are separable, we have the equalities (\cite[Proposition 1.5]{BoSp})
\begin{eqnarray*}
\Lie G_s^\circ = \Lie G_s = \g_s := \{x \in \g \mid \Ad(s) x = x \}; \\ 
\Lie G_{x_s}^\circ = \Lie G_{x_s} = \g_{x_s} := \{x \in \g \mid [x,s] = 0\}.
\end{eqnarray*}

\section{Jordan equivalence} \label{s_jordan}
In this section $G$ is a connected reductive algebraic group.

\subsection{Decomposition classes of a Lie algebra}
The Jordan decomposition in $\g$ allows to define an equivalence relation on $\g$: two elements $x = x_s + x_n$ and $y = y_s + y_n$ of $\g$ are said to be \emph{Jordan equivalent } if and only if there exists $g \in G$ such that $\g_{g \cdot x_s} = \g_{y_s}$ and $ g \cdot x_n = y_n$.
The equivalence classes of $\g$ with respect to this equivalence relation are called \emph{packets} or \emph{decomposition classes (Zerlegungsklassen)} of $\g$.
For $x \in \g$, we denote by $\J(x)$ the packet containing $x$.
If $x = x_s + x_n$, we have $\J(x) = \Ad(G)(\z(\g_{x_s})^\reg + x_n) =: \J(\g_{x_s}, \orb^{G_{x_s}^\circ}_{x_n})$, hence packets are irreducible varieties consisting of isodimensional adjoint orbits.
The packets of $\g$ are indexed by $G$-conjugacy classes of pairs $(\l, \orb^L)$, where $L = G_{x_s}^\circ$ for a semisimple element $x_s \in \g$, and $\orb^L$ is a nilpotent adjoint $L$-orbit in $\l$. 
This implies there are finitely many packets in $\g$; we denote by $\mathcal{J}(\g)$ the set of packets of $\g$.
The definition of a decomposition class was introduced by Borho and Kraft \cite{BK79} for complex reductive Lie algebras;  Spaltenstein \cite{Sp_bad} observed that the definition makes sense under the more general assumptions of this section.

Assume that $p$ is good for $G$, so that $G_{x_s}^\circ$ is a Levi subgroup for all semisimple elements $x_s \in \g$.
Then we can describe the closure (resp. the regular locus of the closure) of a packet in terms of parabolic induction, a procedure parallel to the one defined in \cite{LS79} for the group case.
For $x = x_s + x_n \in \g$ set $L := G_{x_s}^\circ$; we have:
\begin{align*}
\overline{\J(x)} = \bigcup_{x_z \in \z(\l)} \overline{\Ind_{\l}^{\g} \orb^L_{ x_z + x_n} }; \qquad
\overline{\J(x)}^\reg = \bigcup_{x_z \in \z(\l)} \Ind_{\l}^{\g} \orb^L_{ x_z + x_n}.
\end{align*}
It is immediate to check that $\overline{\J(x)} \cap \N(\g) \neq \varnothing$ and $\overline{\J(x)}^\reg \cap \N(\g) = \Ind_{\l}^\g \orb^L_{x_n}$.
The closure description also implies that if $\J \subset \g$ is a packet, both $\overline{\J}$ and $\overline{\J}^\reg$ are unions of packets and the inclusion relation defines a poset structure on the set of packets of $\g$.
We have $\J = \overline{\J}$ if and only if $\J = \J(\g, \{0\}) = \z(\g)$ and
a packet $\J' \subset \overline{\J}^\reg$ satisfies $\overline{\J'} \cap \overline{\J}^\reg = \J'$ if and only if $\overline{\J'}^\reg = \J'$ if and only if $\J' = \z(\g) + \orb$, with $\orb = \overline{\J}^\reg \cap \N(\g)$.

When $p$ is bad for $G$, then $G_{x_s}^\circ$ is not necessarily a Levi subgroup of $G$ for $x_s \in \g$ semisimple.
The author could not find in literature a uniform description of the closure (resp. regular closure) of a decomposition class by means of induction of orbits in such cases.
Nevertheless, some information on packets in Lie algebras of simple groups in bad characteristics can be found in \cite{Sp_bad, Sp_f4}.

\subsection{Jordan classes of a group}
Two elements $g_1 = su, g_2 = rv \in G$ are said to be \emph{Jordan equivalent } if and only if there exists $g \in G$ such that $G_{gs g^{-1}}^\circ = G_{r}^\circ$, $g Z(G_s^\circ)^\circ s g^{-1} = Z(G_r^\circ)^\circ r$ and $ g u g^{-1} = v$.
The equivalence classes of $G$ with respect to this relation are called \emph{Jordan classes} of $G$ and were introduced by Lustzig \cite{LusztigICC}.
We use the notation $J(g)$ for the Jordan class of $g \in G$.
If $g = su $, we have $J(g) = G \cdot ((Z(G_s^\circ)^\circ s)^\reg u) =: J(G_s^\circ, Z(G_s^\circ)^\circ s, \O^{G_s^\circ}_u)$, so that $J(g)$ is irreducible and  contained in a level set.
Jordan classes of $G$ are parametrized by $G$-conjugacy classes of triples $(M, Z(M)^\circ s, \O^M)$, where $M$ is the connected centralizer of a semsimple element in $G$, the connected component $Z(M)^\circ s$ satisfies $C_G(Z(M)^\circ s)^\circ = M$ and $\O^M$ is a unipotent class in $M$.
We denote by $\mathcal{J}(G)$ the set of Jordan classes of $G$; in particular, $\mathcal{J}(G)$  is finite.
The closure (resp. the regular locus of the closure) of a Jordan class can be described via parabolic induction of unipotent classes; for $g = su \in G$ we have:
\begin{align*}
\overline{J(g)} = \bigcup_{z \in Z(G_s^\circ)^\circ s} \overline{G \cdot \left( z\Ind_{G_s^\circ}^{G_z^\circ} \O^{G_s^\circ}_u \right)}; \qquad
\overline{J(g)}^\reg = \bigcup_{z \in Z(G_s^\circ)^\circ s} G \cdot \left( z\Ind_{G_s^\circ}^{G_z^\circ} \O^{G_s^\circ}_u \right).
\end{align*}
One has $\overline{J} \cap \U(G) \neq \varnothing$ if and only if $J = J(L, Z(L)^\circ, \O^L)$ where $L$ is a Levi subgroup.
Closures and regular closures of Jordan classes decompose as unions of Jordan classes and we have a poset structure on the set of Jordan classes, as in the Lie algebra case.
We have $J(g) = \overline{J(g)}$ if and only if $J(g) = Z(G)^\circ \O^G_g$ with $\O^G_g$ a semisimple isolated class, with terminology from \cite{LusztigICC}.
A Jordan class $J'$ contained
in $\overline{J}^\reg$
 is closed therein if and only if $\overline{J'}^\reg
=J'$  if and only if $J' = J'(rv)$ with $r$ isolated.
We refer to \cite[\S 3]{ACE} for further reading, and we remark that Proposition 3.1 therein holds independently of the characteristic of the base field.
In particular we will need the formula for the dimension of the Jordan class $J(su)$:
\begin{equation} \label{dim_form}
\dim J(su) = \dim G \cdot su + \dim Z(G_s^\circ).
\end{equation}
This can be computed as in the proof of \cite[Theorem 5.6 (e)]{CE1}, which works independently of the characteristic of $k$.

We prove an auxiliary result that we will need later.

\begin{Lemma} \label{lem_jnghd}
Let $x \in G$ and let $\pi_G \colon G \to G/\!/G$ be the categorical quotient.
Then $$U(x):= \bigcup_{J \in \mathcal{J}(G), \; x \in \overline{J}} J$$ is a $G$-stable open subset of $G$.

Moreover, if $x$ is semisimple, the subset $U(x)$ is $\pi_G$-saturated.
\end{Lemma}

\begin{proof}
We prove that $ G \setminus U(x) = \bigcup_{x \notin \overline{J}} J$ equals its closure.
Let $ y \in  \overline{\bigcup_{x \notin \overline{J}} J} = \bigcup_{x \notin \overline{J}} \overline{J}$, where the union is finite.
Closures of Jordan classes are unions of Jordan classes, thus $J(y) \subset \bigcup_{x \notin \overline{J}} \overline{J}$ , namely $J(y) \subset \overline{J(z)}$ for some $z \in G$ with $x \notin \overline{J(z)}$.
In particular, $\overline{J(y)} \subset  \overline{J(z)}$ and $x \notin \overline{J(y)}$.
Hence, $y \in J(y) \subset \bigcup_{x \notin \overline{J}} J$.

Assume $x$  semisimple. 
Let $y=su \in U(x)$ and $z = tv \in G$ with $\pi_G(y) = \pi_G(z)$.
We have $J(x) \subset \overline{J(y)}$ if and only if $G_x^\circ \supset G_t^\circ$ and $Z(G_x^\circ)^\circ x \subset Z(G_t^\circ)^\circ t$ hold, up to conjugation by an element of $G$.
Now $s$ is $G$-conjugate to $t$, thus $J(x) \subset \overline{J(z)}$.
\end{proof}

Finally we recall the following facts.
\begin{Lemma}[{{\cite[Corollary 4.5]{ACE}}}] \label{lem_geom}
Let $J \in \mathcal{J}(G)$ and let $x,y \in \overline{J}$ with $J(x) = J(y)$.
Then $(\overline{J}, x) \sim_{\se} (\overline{J}, y)$.
\end{Lemma}

\begin{Remark}[{{\cite[Remark 6.4]{ACE}}}] \label{rmk_geom}
Let $J \in \mathcal{J}(G)$ and let $X = \overline{J}$ or $X = \overline{J}^\reg$.
The locus where $X$ is not smooth (resp. not normal) is closed.
Therefore, to check that $X$ is smooth (resp. normal), it is sufficient to select a point $x_i$ for each  $J_i \in \mathcal{J}(G)$ s.t. $J_i$ is closed in $X$ and to check if $X$ is smooth at $x_i$, by Lemma \ref{lem_geom}.
\end{Remark}

\subsubsection{Sheets}
The set $\mathcal{J}(\g)$ is partially ordered by the relation $\J_1 \preceq \J_2$ if and only if $\J_1 \subset \overline{\J_2}^\reg$.
The sheets of $\g$ with respect to the adjoint $G$-action can be described as the varieties $\overline{\J}^\reg$ for $\J$ maximal in $\mathcal{J}(\g)$.
The arguments in \cite{Borho, BK79} prove that if $k = \C$, then $\overline{\J}^\reg$ is a sheet in $\g$ if and only if $\J = \J(\l, \orb^L)$ with $\orb^L$ rigid in $\l$, i.e., the orbit $\orb^L$ cannot be induced from any proper Levi subalgebra of $\l$.
An analogous result for the sheets of $\g$ holds for $G$ simple simply connected with $p$ very good, as explained in \cite{PremSt}.
The situation for sheets of conjugacy classes of $G$ is  analogous, and treated in detail in \cite{CE1} for the case of $p$ good for $G$. More about sheets in bad characteristic can be found in \cite{Sp_bad, Sp_f4} for the Lie algebra setting and in \cite{ACEbad, Sim} for the group setting.

\section{Reduction to unipotent elements} \label{s_redunip}
In this section we assume that $G$ is a connected reductive group with simply-connected derived subgroup to ensure that centralizers of semisimple elements are connected.
The ideas for the proofs come from \cite[\S 4]{ACE}, so we follow closely the exposition therein and we refer to it quite often.
Some arguments, however, need adjustments under our weaker assumptions on $p$, so we adopt techniques and results from \cite[\S 3]{CES} and \cite[\S 7]{BR}.

\subsection{Restriction to the \'etale locus}
We begin with the characteristic-free version of \cite[Proposition 4.1]{ACE} on the existence of an \'etale neighbourhood of a point in a Jordan class closure.

\begin{Proposition}\label{prop_ACE1}
Let $r \in G$ be semisimple and let $M =G_r$.
There exists a Zariski open neighbourhood $U \subset M$ of $r$ in $M$ with the following features.
\begin{enumerate}[label=(\arabic*)]
\item $U$ is saturated with respect to the categorical quotient $\pi_M \colon M \to M /\!/ M$.
\item For $J_M \in \mathcal{J}(M)$, one has $J_M \cap U \neq \varnothing$ if and only if $r \in \overline{J_M}$.
\item  Consider the map $\gamma \colon G \times^M M \to G$, where $\gamma (g * x) = g x g^{-1}$.
Then $\gamma$ restricts to an \'etale map on $G \times^M U$.
\item $\gamma(G \times^M U) = G \cdot U$ is an open neighbourhood of $r$ in $G$, saturated with respect to the categorical quotient $\pi_G \colon G \to G/\!/G$.
\end{enumerate}
\end{Proposition}

\begin{proof}
We  use Bardsley-Richardson's version of Luna's Fundamental Lemma  \cite[Theorem 6.2]{BR} to restrict $\gamma$ on an open subset of its \'etale locus.
We verify the assumptions therein for the $G$-equivariant morphism $\gamma \colon G \times^M M \to G$ defined by $g * m \mapsto g m g^{-1}$ at the point $e *r \in G \times^M M$.
The orbit $G \cdot (e*r)$  is closed in $G \times^M M$ since it coincides with the image through the quotient $G \times M \to G \times^M M$ of the closed $G$-stable set $G \times \{r\}$; the restriction of $\gamma$ to $G \cdot (e*r)$ is injective.
Since $r$ is semisimple, the conjugacy class $G \cdot r = \gamma(G \cdot (e*r))$ is closed and separable (i.e., the orbit map $G \to G \cdot r$ defined by $g \mapsto g  r g^{-1} $ is separable) by \cite[Proposition 9.1]{Bo}.
By \cite[Proposition A.8.12]{CGP}, the subgroup $M$ is smooth  (as an affine group scheme), hence $G \times^M M$ is smooth too (\cite[proof of Proposition 7.3]{BR}).
We claim that $\gamma$ is \'etale at $e * r$.
Since both the domain and the target set of $\gamma$ are smooth, it is equivalent to prove that $\diff_{e*r}\gamma \colon T_{e *r} (G \times^M M) \to T_r G$ is an isomorphism.
The computation of the transversality of $M$ and $G \cdot r$ at $r$
relies only on semisimplicity of $r$ and is independent of $p$, so that one can follow verbatim  \cite[proof of Proposition 4.1]{ACE}.
In particular, \cite[Theorem 6.2]{BR} applies, hence there exists $f \in k[G]^G$ such that $f(G \cdot r) = 1$ and the restriction of $\gamma$ to $\gamma^{-1}(U_G)$ is \'etale where $U_G := G \setminus \mathcal{Z}(f)$, which is a principal open $G$-saturated set.
Consider the restriction of functions $k[G] \to k[M], f \mapsto \bar f$ and define the principal open $U_M := M \setminus \mathcal{Z}(\bar{f}) \subset M$.
The restriction maps the subring $k[G]^G$ to $k[M]^M$, so that $U_M$ is $\pi_M$-saturated.
The isomorphism $k[G \times^M M]^G \cong k[M]^M$ given in \cite[5.3.3]{BR} yields 
$\gamma^{-1}(U_G) = G \times^M U_M$.

We restrict $U_M$ to construct the sought open set $U$.
Let $U(r) := \bigcup_{J_M \in \mathcal{J}(M), \; r \in \overline{J_M}} J_M$. By Lemma \ref{lem_jnghd} this is a $\pi_M$-saturated open neighbourhood of $r$ in $M$.
Define $U \coloneqq U_M \cap U(r)$, this is an open neighbourhood of $r$ satisfying (1), (2), (3).
We prove that $G \cdot U$ satisfies (4).
Since $U$ is in the \'etale locus of $\gamma$, then  $M_x^\circ = G_x^\circ$ holds for all $x \in U$ thanks to \cite[Remark 4.2]{ACE}.
Let $x= su \in U$ and $y=tv \in G$ s.t. $\pi_G(x) = \pi_G(y) \Leftrightarrow t = gsg^{-1} \in gUg^{-1}$ for some $g \in G$.
Since $U$ is $\pi_M$-saturated, we have $s \in U$ and $t \in gUg^{-1}$.
Moreover  $v \in G_t^\circ \subset g G_s^\circ g^{-1} \subset g M g^{-1}$.
But $U$ is $\pi_M$-saturated, so $v \in gUg^{-1}$ and we conclude $y = tv \in G \cdot U$.
\end{proof}

\begin{Remark}\label{rk_ace}
As in \cite[Remark 4.2]{ACE}, for all $x \in U$ we have
 $\dim G_x = \dim M_x$, equivalently $\dim G \cdot x = \dim G/M + \dim M \cdot x$. 
\end{Remark}

The next result describes the irreducible components of the intersection $\overline{J} \cap U$ for $J \in \mathcal{J}(G)$ meeting $U$ nontrivially. 

\begin{Proposition} \label{prop_irrcomp}
Let $r \in G$ be semisimple and set $M = G_r$.
Build the open $\pi_M$-saturated neighborhood $U \subset M$ of $r$ as in Proposition \ref{prop_ACE1} and retain notation therein for the map $\gamma$.
Let $J \in \mathcal{J}(G)$ with $J \cap U \neq \varnothing$.
Then the following statements hold.
\begin{enumerate}
\item One has $J \cap U = \bigcup J_{M,i} \cap U$ where $J_{M,i} \in \mathcal{J}(M)$ satisfies $J_{M,i} \cap U \cap J \neq \varnothing$.
\item The variety $J \cap U$ is pure of dimension $\dim J - \dim G/M$ and its irreducible components are the closures in $J \cap U$ of $J_{M,i} \cap U$.
\item The map $\gamma_J \colon G \times^M (\overline{J} \cap U ) \to \overline{J} \cap G \cdot U$
defined by $g * x \to gxg^{-1}$ is obtained from $\gamma$ via pull-back, therefore it is \'etale.
\item The varieties $\overline{J_{M,i}} \cap U$ are the irreducible components of $\overline{J} \cap U$.
\end{enumerate}

\end{Proposition}

\begin{proof}
For (1) observe that the inclusion $J \cap U \subset \bigcup J_{M,i} \cap U$ is trivial. Conversely, let $su \in J \cap  J_{M,i} \cap U$ for some index $i$.
Then $Z(M_s^\circ)^\circ s \cap U = Z(G_s^\circ)^\circ s \cap U$ 
and $J_{M,i} \cap U =
M \cdot ( ((Z(M_s^\circ)^\circ s)^\reg \cap U) u)
\subset G \cdot ((Z(G_s^\circ)^\circ s)^\reg u)  = J$, where we used Remark \ref{rk_ace}.

We compute $\dim J_{M,i}$ for $i \in I$.
We can assume $ J_{M,i} = J_M(s_i u)$ for suitable $s_i u \in U$ and we use \eqref{dim_form} and Remark \ref{rk_ace}.
Then $\dim J_{M,i} = \dim M \cdot (s_i u) + \dim Z(M_{s_i}^\circ) = \dim G/M + \dim G \cdot (s_i u) + \dim Z(G_{s_i}^\circ)$ and we conclude, since all $s_i u \in J$.
Assertion (2) follows directly.

We turn to (3): we use the argument in the first part of the proof of \cite[Proposition 7.5]{BR}. In this terminology, Proposition \ref{prop_ACE1} states that $U$ is an \'etale slice for the $G$-action on $G$ at $r$.
The subvariety $\overline{J} \subset G$ is $G$-stable, closed and contains $r$.
Then  $\overline{J} \cap U$ is an \'etale slice for the $G$-action on $\overline{J}$ at $r$, and in particular
 the following is a pull-back diagram:
 \begin{center}
\begin{tikzcd}
G \times^M (\overline{J} \cap U) \arrow[d]  \arrow[r, "\gamma_J"] & \overline{J} \cap G \cdot U  \arrow[d]\\
G \times^M U \arrow[r, "\gamma"] & G \cdot U 
\end{tikzcd}
\end{center}
where $\overline{J} \cap G \cdot U \to \overline{J}$ is the open immersion.

Finally we prove (4).
 Consider the map $\gamma_J \colon G \times^M (\overline{J} \cap U ) \to \overline{J} \cap G \cdot U$.
Let $X$ be an irreducible component of $\overline{J} \cap U$, then it is $M$-stable by connectedness of $M$.
Moreover, the restriction of $\gamma$ induces a dominant morphism $G \times^M X \to \overline{J}$ because
by \cite[I, 7.1]{Hartshorne} $\dim X \geq \dim U + \dim \overline{ J} - \dim G = \dim J - \dim G/M$.
By the \'etale property there is an open subset of the image $G \cdot X$ of dimension $\dim (G \times^M X) = \dim G/M + \dim X \geq \dim J$.
The inequality $\dim G \cdot X \leq \dim J$ implies 
$\dim G \cdot X = \dim J = \dim \overline{J}$ and $G \cdot X$ is an irreducible subset of $\overline{J}$, also irreducible, concluding $\overline{G \cdot X} = \overline{J}$.
It follows that $J$ is open in $\overline{G \cdot X}$, hence $J \cap G \cdot X \neq \varnothing$ and $J \cap X \neq \varnothing$.
Since $J$ is open in its closure, we have that $J \cap X = J \cap X \cap U$ is open in $\overline{J} \cap X \cap U = X$, hence $J \cap X$ is irreducible.
Clearly $J \cap X \subset J \cap U = \bigcup J_{M,i} \cap U$, hence there exists $i$ such that
$J \cap X \subset \overline{J_{M,i} \cap U}^{J \cap U} \subset \overline{J_{M,i}}$.
We get $\overline{X} = \overline{J \cap X} \subset \overline{J_{M,i}}$, therefore 
$X \subset \overline{J_{M,i}} \cap U \subset \overline{J} \cap U$.
Since $X$ is an irreducible component of $\overline{J} \cap U$ 
and $\overline{J_{M,i}} \cap U$ is irreducible, we conclude.
\end{proof}

\subsection{Collection of the main results}
In this part we collect the characteristic-free versions of the  results in \cite[\S 4]{ACE}, starting from the counterpart of \cite[Proposition 4.3]{ACE}.

\begin{Theorem}\label{th_ACE}
Let $J \subset G$ be a Jordan class.
Let $rv \in \overline{J}$ and set $M := G_r$.
Then 
\begin{equation} \label{eq_clos}
(\overline{J}, rv) \sim_{\se} \left( \bigcup_{i \in I_{J, rv}} r^{-1} \overline{J_{M,i}}, v \right)
\end{equation}
where $I_{J, rv}$ is an index set for $\{ J_{M,i} \in \mathcal{J}(M) \mid J_{M,i} \cap J \neq \varnothing, rv \in \overline{J_{M,i}} \}$.

Moreover, if $rv \in \overline{J}^\reg$, then $rv \in \overline{J_{M,i}}^\reg$ for all $i \in I_{J, rv}$ and  
\begin{equation}  \label{eq_reg}
(\overline{J}^\reg, rv) \sim_{\se} \left( \bigcup_{i \in I_{J, rv}} r^{-1} \overline{J_{M,i}}^\reg, v \right).
\end{equation}
\end{Theorem}

\begin{proof}
Build the $\pi_M$-saturated open neighbourhood $U$ and the \'etale map $\gamma_J \colon G \times^M (\overline{J} \cap U) \to \overline{J} \cap G \cdot U$ as in Propositions \ref{prop_ACE1} and \ref{prop_irrcomp}.
Let $x \in \overline{J} \cap U$.
We show that $(\overline{J}, x) \sim_{\se} (\overline{J} \cap U, x)$.
By Proposition \ref{prop_irrcomp} (3) the map $\gamma_J$ is \'etale, yielding a chain of smooth equivalences:
$$ (\overline{J}, x) \sim_{\se} (\overline{J} \cap G \cdot U, x) \sim_{\se} ( G \times^M (\overline{J} \cap U), e * x) \sim_{\se} (G \times (\overline{J} \cap U), (e,x)) \sim_{\se} (\overline{J} \cap U, x).$$

We apply (4) from Proposition \ref{prop_irrcomp} and we obtain
$$(\overline{J} \cap U, x) \sim_{\se}
 \left( \bigcup_{i \in I_{J,r}} \overline{J_{M,i}} \cap U, x \right) \sim_{\se} 
\left( \bigcup_{i \in I_{J,r}} \overline{J_{M,i}} , x \right),
$$
where $I_{J,r}$ is an index set for $J_M \in \mathcal{J}(M)$ such that $r \in \overline{J_M}$.
If we define the $M$-stable open subset $U(x) := \bigcup_{J_M \in \mathcal{J}(M) , \; x \in \overline{J_M}}J_M$ as in Lemma \ref{lem_jnghd}, we can shrink $U$ and get  
$(\overline{J} \cap U, x) \sim_{\se}
(\overline{J} \cap U \cap U(x), x)$.
Taking $x = rv$ and composing with the left translation by $r^{-1} \in Z(M)$, we derive \eqref{eq_clos}.

Finally \eqref{eq_reg} is obtained using Proposition \ref{prop_ACE1} (3) and Remark \ref{rk_ace} with the same argument in the final part of the proof of \cite[Proposition 4.3]{ACE}.
\end{proof}

We can now state the main consequences of Theorem \ref{th_ACE}.
\begin{Corollary}[{{\cite[Corollaries 4.6 and 4.7]{ACE}}}]
Let $J \in \mathcal{J}(G), rv \in \overline{J}$ and $M = G_r$. Define $I_{J, rv}$ as in Theorem \ref{th_ACE}.
\begin{enumerate}
\item $\overline{J}$ is unibranch, resp. normal, resp. smooth at $rv$ if and only if $\lvert I_{J, rv} \rvert = 1$ and $r^{-1} \overline{J_{M,i}}$ is unibranch, resp. normal, resp. smooth at $v$.
\item Assume $\lvert I_{J, rv} \rvert = 1$. Then $\overline{J}$ is Cohen-Macaulay  at $rv$ if and only if  and $r^{-1} \overline{J_{M,i}}$ is Cohen-Macaulay  at $v$.
\end{enumerate}
\end{Corollary}

The combinatorial parametrization in \cite[\S 4]{ACE} of the set $I_{J,rv}$ from Theorem \ref{th_ACE} in terms of the Weyl group of $G$ carries over straightforwardly to the characteristic-free case.
In this manuscript we will not apply this parametrization, so we address the interested reader to the mentioned paper.

\section{Bardsley-Richardson's logarithm-like map} \label{s_br}
In this section $G$ denotes a connected reductive group (we drop the assumption on simple-connectedness).

\begin{Definition} \label{def_loglike}
 A $G$-equivariant morphism of varieties $\lambda \colon G \to \g$ 
 is called a \emph{log-like map (\`a la Bardsley-Richardson) for $G$} if 
there exist an open neighbourhood $U_\lambda$ of $\mathcal{U}(G)$ and an open neighbourhood $V_\lambda$ of $\mathcal{N}(\g)$, both saturated with respect to their categorical quotients by $G$, 
such that
the diagram
 \begin{equation} \label{eq_loglikecd}
\begin{tikzcd}
U_\lambda \arrow[d, "\pi_G"']  \arrow[r, "\lambda"] & V_\lambda \arrow[d, "\pi_G"]\\
U_\lambda/\!/ G   \arrow[r] & V_\lambda /\!/ G 
\end{tikzcd}
\end{equation}
is Cartesian with \'etale surjective horizontal arrows\footnote{The vertical arrows denote the respective categorical quotients of $U_\lambda$ and $V_\lambda$ by $G$, the bottom horizontal arrow is determined by the universal property of the categorical quotient.}.
\end{Definition}

In particular, a log-like map induces a Springer isomorphism $\lambda_{\mid \U(G)} \colon \U(G) \to \N(\g)$.
The existence of a log-like map is not always ensured under the general assumption on $G$ made in this section.
When $p$ is bad, the set of unipotent classes of $G$ is often not in  bijective correspondence with the set of nilpotent adjoint orbits of $\g$, see \cite{LiebeckSeitz}.  
Existence problems may occur even under the assumption that $p$ is good: the unipotent variety of $\PSL_2(k)$ and the nilpotent variety of $\mathfrak{psl}_2(k)$ are not isomorphic if $p = 2$ (see  \cite[Appendix]{SobTay}).
For a counterexample with $p$ separably good, see section \ref{ss_sep}.

Nevertheless, log-like maps can be constructed under some (quite general) sufficient conditions.

\begin{Example}\label{ex_loglike}
We recall the main examples from \cite[\S 9.3]{BR}.
\begin{enumerate}
\item If $G = \GL_n(k)$, then $\lambda \colon G \to \g, g \mapsto  g - I_n$ is a log-like map.
In this case we can take $V_\lambda  = \{x \in \g \mid \det x \neq 1\}$ and $U_\lambda = G$; the map $\lambda$ induces and isomorphism of varieties onto its image.
If $g = su \in G$, then $\lambda(s)$ is the semisimple part of $\lambda(g)$ and $\lambda(g) - \lambda(s)$ is the nilpotent part of $\lambda(g)$.
\item Suppose the reductive group $G$ admits a  representation
$\rho \colon G \to \GL(V)$ satisfying the following properties:
\begin{itemize}
\item[(BR1)] the differential at the identity $\diff_e \rho \colon \g \to \gl(V)$ is injective, and
\item[(BR2)] there is a $G$-stable subspace $\a \subset \gl(V)$ such that $\gl(V) = \a \oplus \diff_e \rho(\g)$ and  $\id_V \in \a$.
\end{itemize}
Then one can obtain a log-like map $\lambda \colon G \to \g$ via the following composition:
$$ G \xrightarrow{\rho} 
\GL(V) \hookrightarrow \gl(V) 
\xrightarrow{\pr} \diff_e \rho(\g) \cong \g$$
where $\pr$ is the projection on $\diff_e \rho(\g)$ with kernel $\a$.
A class of groups  satisfying (BR1) and (BR2) is given by simple simply connected groups with $p$ very good.
\end{enumerate}
\end{Example}

In view of the reduction argument in section \ref{s_redunip}, for the scope of this manuscript it suffices to build a 
log-like map for all centralizers of semisimple elements in $H$ reductive with $[H,H]$ simply connected.

\begin{Remark} \label{rmk_mcninch}
McNinch exhibited a log-like map for any Levi subgroup of a group $H := H_1 \times S$ with $S$ a torus  and $H_1$  semisimple simply connected s.t. $p$ is very good for $H_1$, \cite[Theorem 54]{McN_NOgood}.
With a similar procedure one can obtain a log-like map  for any \emph{strongly standard} reductive group, i.e. a group $K$ separably isogenous to $C_H(D)$ with $D \leq H$ a subgroup of multiplicative type, see \cite[\S 2.4]{McNTes}.
Thanks to \cite[Lemma 17]{McNS} which holds independently of $p$,  all centralizers of semisimple elements of $H$ can be obtained as $C_H(D)$ with $D \leq H$ of multiplicative type.
\end{Remark}

After a private communication with George McNinch, the author could also observe the following.

\begin{Remark} \label{rmk_ambrosiomcninch}
Let $\rho \colon G \to \GL(V)$ and $\lambda \colon G \to \g$ be as in Example \ref{ex_loglike} (2).
Let $s \in G$ be semisimple and put $M:=G_s^\circ$, so that $\m = \g_s$.
We claim that $M$ satisfies again the properties (BR1-BR2). 
Consider $\bar \rho \coloneqq \rho_{\vert_M} \colon M \to \GL(V)$.
It is immediate to check that $\diff_e \bar \rho \colon \m \to \gl(V)$ is injective.
Moreover, we have a decomposition of $M$-stable vector spaces $\g = \m \oplus \image \phi$ where $\phi \coloneqq \Ad(s) - \id_\g \colon \g \to \g$.
Recall that $\gl(V) = \diff_e \rho(\g) \oplus \a$ is a $G$-stable decomposition with $\id_V \in \a$.
Then $\bar \a := \diff_e \rho (\image \phi) \oplus \a$ is the sought $M$-stable complement of $\diff_e \bar \rho(\m)$ in $\gl(V)$.
This proves the claim.

Since $\lambda$ is $G$-equivariant and $M$ is a stabilizer, one has
$\lambda(M) \subset \m$.
 Namely $\pr \circ \rho(M)\subset \m$, whence
the log-like map for $M$ obtained applying the construction by Bardsley-Richardson to the group $M$ agrees with  the restriction of $\lambda$ to $M$.
\end{Remark}

\subsection{Compatibility between log-like maps and Jordan stratification} \label{ss_compatible}
In this section $G$ is a connected reductive group.
Any $G$-equivariant morphism $\lambda \colon G \to \g$ satisfies the following  relations, for all $h \in G$:
 \begin{align}
 G_h \subset G_{\lambda(h)}, \label{incl_1}\\
 \lambda(G_h) \subset \g_h. \label{incl_2}
 \end{align}

\begin{Lemma} \label{lem_ss}
A $G$-equivariant morphism $\lambda \colon G \to \g$ maps semisimple elements of $G$ to semisimple elements of $\g$.
Moreover, if $su$ is the Jordan decomposition of $g \in G$, then $\lambda(s) + (\lambda(g) - \lambda(s))$ is the Jordan decomposition of  $\lambda (g) \in \g$.
\end{Lemma}

\begin{proof}
Let $s \in G$ be semisimple.
Then it is contained in a maximal torus $Z \leq G$ and there exists $t \in T$ such that $G_t^\circ = T$.
Hence by $G$-equivariance of $\lambda$  we have
$\lambda(s) \in \lambda(T) = \lambda(G_t^\circ) \subset  \g_t = \Lie  G_t^\circ = \Lie T$, i.e. $\lambda(s)$ lies in the Lie algebra of a torus of $G$, i.e., the element $\lambda(s)$ is semisimple.\footnote{
The argument actually proves that any maximal torus $T \leq G$ satisfies $\lambda(T) \subset \t$.}  

We consider the reductive subgroup $G_s^\circ$ acting via the adjoint action on $\g_s$. 
We have $s \in \overline{G_s^\circ \cdot (su)}$, whence $\lambda(s) \in \overline{G_s^\circ \cdot \lambda(su)}$.
Now $\lambda(s)$ is semisimple by the first part.
The variety $\overline{G_s^\circ \cdot \lambda(su)} $  contains exactly one closed orbit, which is in turn $G_s^\circ \cdot \lambda(s) = \{\lambda(s)\}$.
We conclude that $\lambda(s)$ is the semisimple part of $\lambda(su)$.
\end{proof}

From now until the end of section \ref{ss_compatible} we assume that $G$ admits a log-like map and we fix a triple $(\lambda, U_\lambda, V_\lambda)$
where $\lambda \colon G \to \g$ is a log-like map and $U_\lambda, V_\lambda$ satisfy \eqref{eq_loglikecd}.

\begin{Remark}
Let $\J \subset \g$ be a decompostion class. 
By the description of the closure of a Jordan class, we have $\overline{\J} \cap \N(\g) \neq \varnothing$, so that $\overline{\J} \cap V_\lambda \neq \varnothing$, whence $\J \cap V_\lambda \neq \varnothing$, by openness of $\J$ in $\overline{J}$. 
This implies that $V_\lambda$ meets all Jordan classes of $\g$.

Let $J \subset G$ be a Jordan class of $G$ s.t. $\overline{J} \cap \U(G) \neq \varnothing$.
Then a similar argument implies $J \cap U_\lambda \neq \varnothing$.
We recall that $\overline{J} \cap \U(G) \neq \varnothing$ if and only if $J = J(su)$ with $G_s^\circ$ is a Levi subgroup of $G$ and $s \in Z(G_s^\circ)^\circ$.
In general, the open set $U_\lambda$ need not meet all Jordan classes of $G$ nontrivially, see section  \ref{sub_outside}.
\end{Remark}

\begin{Remark}\label{rk_etale}
For any element $h \in  U_\lambda$, thanks to  \eqref{incl_1} and since  \'etale maps preserve dimensions, we have $G_h^\circ = G_{\lambda(h)}^\circ$.
If moreover the subgroup $G_{\lambda(h)}$ is itself connected, then $G_h^\circ =G_h =   G_{\lambda(h)}$ (in particular, if $h$ is semisimple and $p$ is not torsion for $G$.)
Under separability assumptions (for example if $h$ is semisimple), we also have $\g_h = \g_{\lambda(h)}$.
\end{Remark}

The next result proves that log-like maps  preserve Jordan classes on the fixed open \'etale locus.
\begin{Proposition} \label{prop_respect}
Let $\lambda \colon G \to \g$ be a log-like map and let $su \in U_\lambda$.
The following assertions hold.
\begin{enumerate}[label=(\arabic*)]
\item $\lambda( J(su) \cap U_\lambda) \subset \J(\lambda(su)) \cap V_\lambda$.
\item If the identity element $e \in \overline{J(su)}$, then $\J(\lambda(su)) = \J(\lambda(s) + \lambda(u))$.
\end{enumerate}
\end{Proposition}
\begin{proof}
We prove (1). For $su \in G$ put $G_s^\circ =: M, Z(M) =: Z$ so that $\m = \g_s$.
It suffices to show that $\lambda(((Z^\circ s)^\reg \cap U_\lambda) u) \subset \J(\lambda(su))$.
For $z \in (Z^\circ s)^\reg \cap U_\lambda$, we  show that $\lambda(zu)$ is Jordan equivalent to $\lambda(su)$.
Recall that we know the Jordan decomposition of such elements thanks to Lemma  \ref{lem_ss}.

We first deal with the semisimple part.
We have $M = G_z^\circ$ and thanks to  Remark \ref{rk_etale} we derive also $M = G_{\lambda(s)}^\circ = G_{\lambda(z)}^\circ$.
In particular $\m = \g_{\lambda(z)} = \g_{\lambda(s)}$.

For the nilpotent part, we define  the $M$-equivariant morphism 
\begin{align*}
\nu \colon M \times Z^\circ s &\to \N(\m) \\
(g, z) & \mapsto g \cdot (\lambda(zu) - \lambda(z)) = \lambda(z gug^{-1}) - \lambda(z),
\end{align*}
which is well-defined by Lemma \ref{lem_ss}, in particular the Jordan decomposition $\lambda(zu) = \lambda(z) + \nu(e,z)$ holds in $\m$ for all $z \in Z^\circ s$.
Here $M$ acts on the domain of $\nu$ by left multiplication on the first factor and on the codomain by the adjoint action.
The closure $\overline{\image \nu}$ is an irreducible $M$-stable subset of $\N(\m)$, hence $\overline{\image \nu} = \overline{\orb}$ for some nilpotent $M$-orbit $\orb \subset \m$.
Fix $d := \dim \orb$, then $\orb = \overline{\image \nu} \cap \m_{(d)}$ and $M$-equivariance implies $\orb \subset \image \nu$ so that $\orb = \image \nu \cap \m_{(d)}$.
We prove that $\varnothing \neq \nu (M \times ( Z^\circ s)^\reg \cap U_\lambda )) \subset \m_{(d)}$.
We  have $s \in (Z^\circ s)^\reg \cap U_\lambda$, proving non-emptiness.
Let $z \in (Z^\circ s)^\reg \cap U_\lambda$.
Then the \'etale property of $\lambda$ on $U_\lambda$ implies: $$\dim G \cdot (\lambda(z) + \nu(e,z)) = \dim G \cdot (zu) = \dim G \cdot (su) = \dim G \cdot (\lambda(s) + \nu(e,s)).$$
This equality yields 
$\dim G_{\lambda(z)} \cdot \nu(e,z) = \dim G_{\lambda(s)} \cdot \nu(e,s)$ which in turn implies
$\dim M \cdot \nu(e,z) = \dim M \cdot \nu (e,s)$, whence
$M \cdot \nu(e,z) \subset \m_{(d')}$
with $d' = \dim G \cdot (su) - \dim G \cdot s = \dim M \cdot u$.
The set $(Z^\circ s)^\reg \cap U_\lambda$ is open in $Z^\circ s$ so $\nu(M \times ((Z^\circ s)^\reg \cap U_\lambda))$ is dense in $\image \nu$ and must meet $\orb$ nontrivially, which implies $\orb = \nu(M \times ((Z^\circ s)^\reg \cap U_\lambda))$ so that $d' = d$ and we conclude.
In particular $\J(\lambda(su)) = \J(\m, \orb)$.

To prove (2), recall that $J(su)$ closes at the identity if and only if $Z^\circ s = Z^\circ$.
We already observed in (1) that $d' = \dim M \cdot u$.
Hence $d' = \dim M \cdot \lambda(u)$ since $\lambda$ induces an isomorphism $\U(M) \to \N(\m)$.
Finally we have $\lambda(u) = \nu(e,e) \in \overline{\image \nu}$, yielding
$M \cdot \lambda(u) \subset \overline{\image \nu} \cap \m_{(d')} = \orb$, hence $\J(\lambda(su)) = \J(\lambda(s) + \lambda(u))$.
\end{proof}

Two of our main results immediately follow.
\begin{Corollary} \label{cor_respect}
The restriction of the log-like map $\lambda_{\mid U_\lambda} \colon U_\lambda \to V_\lambda$ induces a map
 $$ \Lambda \colon \{  J \in \mathcal{J}(G) \mid J \cap U_\lambda \neq \varnothing \} \to \{ \J \in \mathcal{J}(\g) \mid \J \cap V_\lambda \neq \varnothing \}$$
 via the assignment $\Lambda(J(g)) = \J(\lambda(g))$ for $g \in U_\lambda$.

For $J, J'$ in the domain of $\Lambda$ and $n \in \mathbb{N}$, we have $\dim J = \dim \Lambda(J)$, $J \subset G_{(n)}$ if and only if $\Lambda(J) \subset \g_{(n)}$ and $J' \subset \overline{J}$ if and only if $\Lambda(J') \subset \overline{\Lambda(J)}$.

For a Levi subgroup $L \leq G$ and $u \in  \mathcal{U}(L)$, we have $\lambda(\Ind_L^G \O^L_u) = \Ind_{\l}^{\g} \orb^L_{\lambda(u)}$.
\end{Corollary}

\begin{Corollary} \label{cor_se}
Let  $u \in \U(G)$ and let $J= J(L, Z(L)^\circ, \O^L_v)$ be a Jordan class in $G$ s.t. $u \in \overline{J}$, resp. $u \in \overline{J}^\reg$. Then 
$(\overline{J}, u) \sim_{\se} (\overline{\J(\l, \orb^L_{\lambda(v)})}, \lambda(u))$, resp. $(\overline{J}^\reg, u) \sim_{\se} (\overline{\J(\l, \orb^L_{\lambda(v)})}^\reg, \lambda(u))$.
\end{Corollary}

\subsubsection{Analogies and differences in positive characteristic} \label{sss_summary}
Assume that $G$ has simply connected derived subgroup and admits a log-like map and let $J \in \mathcal{J}(G)$. 
Then one can perform a local study of a point $rv$ in $\overline{J}$ or $\overline{J}^\reg$ first applying Theorem \ref{th_ACE} and then Corollary \ref{cor_se} and Remarks \ref{rmk_mcninch} or \ref{rmk_ambrosiomcninch}, thus generalizing \cite{ACE}.
Nevertheless, 
two obstacles arise in the concrete application of the procedure.
Firstly, in \cite{ACE} we could apply the parametrization in terms of the Weyl group of $G$ to deduce a sufficient condition to have $\lvert I_{J, rv} \rvert = 1$, cf. the notation $I_{J, rv}$  from section \ref{s_redunip}.
Unfortunately the proof of \cite[Lemma 4.8]{ACE} relies on the use of analytic isomorphisms and on the fact that the categorical quotient behaves well with respect to closed $G$-stable subvarieties, unavailable in the general characteristic-free setting.
Secondly, the geometry of Jordan classes of Lie algebras of reductive groups in positive characteristic has not  been investigated in general yet.

To conclude we gather some applications, examples and remarks about log-like maps. 
As expected, the cases $\GL_n(k)$ with $p$ arbitrary and $\SL_n(k)$ with $(p,n) = 1$ are well-behaved. 
In section \ref{ss_sep} we show that for $p = 2$, closures of Jordan classes in $\SL_2(k)$  fail to mirror geometric properties of decomposition classes closures in $\sl_2(k)$. 
This suggests that the (quite common) assumptions required to obtain the main results in the manuscript cannot be relaxed too much.

\subsection{The general linear group} \label{sub_GLn}
Let $G = \GL_n(k)$ and $p$ be arbitrary. 
By means of our methods we can reobtain the following result from \cite[3.3]{LusztigCCRG}, see also \cite{BroerDecoVar, Peterson}.

\begin{Theorem} \label{thm_GLn}
Let $G = \GL_n(k)$ and suppose  $J = J(L, Z(L), \O^L) \in \mathcal{J}(G)$.
\begin{enumerate}
\item The closure $\overline{J}$ is normal if and only if $[L,L]$ is of type $d \sf{A}_\ell$ with $d \ell = n$ and $\O^L$ has the same Jordan normal form in each one of the $d$ blocks.
\item Suppose furthermore that $\overline{J}^\reg$ is a sheet in $G$. Then it is smooth.
\end{enumerate}
\end{Theorem}

\begin{proof}
Let $\lambda$ be the log-like map  from Example \ref{ex_loglike} (1).
For (1), remark that there is a unique closed Jordan class in $\overline{J}$, namely $J(e) = J(G, Z(G), \{e\})$. By Remark \ref{rmk_geom} it is enough to check normality of $\overline{J}$ at $e$.
Then Corollaries \ref{cor_respect} and \ref{cor_se} imply that $(\overline{J}, e) \sim_{\se} (\overline{\J}, 0)$ where $\J = \J(\l, \lambda(\O^L))$. We conclude by \cite[Proposition 9.2]{BroerDecoVar}.

For (2), we have $\O^L = \{e\}$ and the unique closed Jordan class in $\overline{J}^\reg$ is $J(G, Z(G), \O^G)$, where $\O^G := \Ind_L^G \{e\}$. Let $u \in \O^G$,
by Corollaries \ref{cor_respect} and \ref{cor_se}  we have $(\overline{J}^\reg, u) \sim_{\se} (\overline{\J}^\reg, \lambda(u))$ where $\J = \J(\l, \{ 0 \})$.
 Finally $\overline{\J}^\reg$ is a sheet in $\g$, hence smooth, by \cite[Proposition 9.2]{BroerDecoVar}.
\end{proof}

\subsection{The special linear group in very good characteristic} \label{sub_SLn}
Let $G = \SL_n(k)$ with $(n,p) = 1$  and consider the log-like map $\lambda \colon G  \to \g$ given by $g \mapsto g - \frac{\tr g}{n}I_n$, as per Example \ref{ex_loglike} (2).

The map $\lambda$ is flat by miracle flatness (\cite[Theorem 23.1]{Matsumura}): domain and codomain are smooth varieties and $\lambda$ has finite fibres.
Computing the \'etale locus amounts therefore to determining the ramification locus.
The degree of the map is generically $n$: for $x \in \g$ one has $\lambda^{-1}(x) = \{ x + t I_n \mid \det(x + t I_n) = 1\}$.
The polynomial $p_x := \det(x + t I_n) - 1 \in k[t]$ is monic of degree $n$ and admits $n$ distinct solutions if and only if the resultant $R := \Res(p_x(t), p'_x(t)) \neq 0$.
Remark that $R \in k[\g]^G$,
hence $R(x) \neq 0$ if and only if $\lambda^{-1}(x)$ consists of exactly $n$ elements.
Since $G$ and $\g$ are irreducible varieties, the ramification locus of $\lambda$ is determined by the vanishing locus of $R$ in $\g$.
So we can choose $U_\lambda = G \setminus \mathcal{Z}(\lambda^*(R))$ and $V_\lambda = \g \setminus \mathcal{Z}(R)$ so that $\lambda_{\mid U_\lambda} \colon U_\lambda \to V_\lambda$ is an $(n:1)$ \'etale covering.

Observe that $\lambda^{-1}(0) = Z(G)$.
By saturation of $U_\lambda$, we deduce $Z(G) \mathcal{U}(G) \subset U_\lambda$.
Thanks to the description of closures of Jordan classes, we observe that each $J \in \mathcal{J}(G)$ satisfies $\overline{J} \cap Z(G) \neq \varnothing$, whence  $J \cap U_\lambda \neq \varnothing$.
Let $su \in U_\lambda$ and set $M := G_s$ and $Z := Z(M)$; then $M$ is a Levi subgroup of $G$ and $Z^\circ s = Z^\circ z$ for some $z \in Z(G)$, see for example \cite[Remark 3.7]{AC}.

We have $\lambda(J(M, Z(M)^\circ s, \O^M_u) \cap U_\lambda)  \subset \J(\m, \orb^M_{\lambda(u)})$.
Indeed, the Jordan class $J = z J(M, Z^\circ, \O^M_u)$ closes at $z$ and
since $\lambda(z) = 0$ we have $\lambda(zu) = zu - \frac{\tr(z)}{n} I_n = zu - z = z(u-I_n) \in k^\times (u-I_n) = k^\times \lambda(u) \subset M \cdot \lambda(u)$.
With the  dimension argument used to prove Proposition 
\ref{prop_respect} one obtains the claim.

\begin{Proposition} \label{prop_sln}
Let $G = \SL_n(k)$ and $(p,n) =1$.
Let $J = J(L, Z(L)^\circ z, \{e\}) \in \mathcal{J}(G)$.
The sheet $\overline{J}^\reg$ is smooth (resp. normal) if and only if the  sheet $\overline{J(\l, \{0\})}^\reg$ in $\g$ is smooth (resp. normal).
\end{Proposition}
\begin{proof}
Since $L$ is a Levi subgroup we can assume $z \in Z(G)$.
As in the proof of Theorem \ref{thm_GLn} (2), the unique closed Jordan class in $\overline{J}^\reg$ is $J(G, Z(G)^\circ z, \Ind_L^G \{e\}) = z \Ind_L^G \{e\}$.
Fixing $\lambda$ as in the current subsection and $u \in \Ind_L^G \{e\}$ we have $(\overline{J}^\reg, zu) \sim_{\se} (\overline{\J (\l, \{0 \})}^\reg, \lambda(u))$, using Corollaries \ref{cor_respect}, \ref{cor_se} and the preliminary discussion in this subsection.
\end{proof}

\subsection{Outside the \'etale locus} \label{sub_outside}
In this subsection we assume that $p$ is good for $G$. Then for a semisimple $x_s \in \g$, the connected centralizer $G_{x_s}^\circ$ is a Levi subgroup of $G$.
Let $s \in G$ semisimple with $G_s^\circ$ not a Levi subgroup\footnote{Such subgroups always occur unless $[G,G]$ is a product of simple subgroups of type $\sf A$.}.
Put $M := G_s^\circ, Z := Z(M)$, and define the group $L := C_G(Z^\circ)$; then $L$ is the minimal Levi subgroup of $G$ containing $M$.
Hence the chain of inclusions $M \subsetneq L \subset  G_{\lambda(s)}^\circ$ holds.
In this situation, we obtain Jordan classes in $G$ which never meet the locus of $G$ where $\lambda$ is \'etale, thus $\lambda$ does not behave well with respect to the Jordan stratification outside this locus, as the following example illustrates.

\begin{Example}
Let $p \neq 2$ and  $G =\Sp_4(k) = \{ A \in \GL_4(k) \mid A^T J A = J \}$ where $J = \left( \begin{smallmatrix} 0 & I_2 \\ -I_2 & 0 \end{smallmatrix} \right)$, then a log-like map $\lambda \colon G  \to \g$ can be written down explicitly via the procedure illustrated in Example \ref{ex_loglike} (2), where we choose $\rho \colon G \to \GL_4(k)$ to be the natural representation.
Let $T$ be the  torus of diagonal matrices in $G$, the restriction of $\lambda$ to $T$ acts as follows
\begin{align*}
 \left( \begin{smallmatrix} a & 0 & 0 & 0 \\ 
0 & b & 0 & 0 \\ 
0 & 0 & a^{-1} & 0 \\ 
0 & 0 & 0 & b^{-1}  \end{smallmatrix} \right) \mapsto  
 \frac{1}{2} \left(
\begin{smallmatrix}a-a^{-1} & 0 & 0 & 0 \\ 
0 & b-b^{-1} & 0 & 0 \\ 
0 & 0 & a^{-1}-a & 0 \\ 
0 & 0 & 0 & b^{-1}-b  \end{smallmatrix} \right)
\end{align*}
where $a, b \in k^\times$.

In particular, $G$ admits a nontrivial semisimple isolated conjugacy class $G \cdot s = J(s) = \overline{J(s)}$ with $s := \left(
\begin{smallmatrix} 1 & 0 & 0 & 0 \\ 
0 & -1 & 0 & 0 \\ 
0 & 0 & 1 & 0 \\ 
0 & 0 & 0 & -1 \end{smallmatrix} \right)$.
The centralizer  $G_s^\circ = G_s \cong \SL_2(k) \times \SL_2(k)$ is not a Levi factor in $G$.
Nonetheless $\lambda(s) = 0$ and $\lambda(J(s)) = \J(\lambda(s)) = \{0\}$. Observe that $G_{\lambda(s)} = G_{\lambda(s)}^\circ = G$, which is the minimal Levi subgroup of $G$ containing $G_s^\circ$.
The  element $s$ (hence $G \cdot s$) is not contained in the \'etale locus of $\lambda$, since $\dim G \cdot s > \dim  \lambda(G \cdot s)$.
\end{Example}

It would be interesting to know if the following more general question has a positive answer.
Assume $G$ be a reductive group, $p$ is good for $G$ and suppose there exists a log-like map  $\lambda \colon G \to \g$.
If $s \in G$ is semisimple, then is $G_{\lambda(s)}^\circ$ the minimal Levi subgroup in $G$ containing $G_s^\circ$?

\subsection{A counterexample in separably good characteristic} \label{ss_sep}
Suppose $p = 2$ and $G = \SL_2(k)$, let $T$ be the torus  of diagonal matrices in $G$.
Then $\g = \sl_2(k)$ is a nilpotent Lie algebra.
The centre $\z(\g)$ coincides with the set of semisimple elements of $\g$.
We observe that $p$ is separably good for $G$, so that $G$ admits a Springer isomorphism by \cite{Springer_unipotent}. 

Jordan classes of $\g$ coincide with sheets, which in turn coincide with the level sets.
We have
$\g_{(0)} = \z(\g)$  and $\g_{(2)} = \g \setminus \z(\g) = \z(\g) + \orb$, where $\orb = \N(\g)^\reg$ is the regular nilpotent orbit of $\g$.
The closures $\overline{\g_{(0)}} = \g_{(0)}$ and $\overline{\g_{(2)}} = \g$ are smooth varieties.
On the other hand, the regular unipotent class $\O = \U(G)^\reg$ is a Jordan class in $G$ and $(\overline{\O}, e) = (\U(G), e)$ is not smooth.
So there is no  candidate $\J \in \mathcal{J}(\g)$ satisfying  $(\U(G) , e) \sim_{\se} (\overline{\J}, 0)$.
In particular, there exists no log-like map for $G$.

\bibliographystyle{abbrv}
\bibliography{biblio}
\end{document}